\documentclass[11pt]{article}
\usepackage{latexsym,amssymb}

\usepackage{multirow}
\usepackage{url}

\usepackage{amsmath, amsthm,epsfig,bm,booktabs,amssymb}
\usepackage{algorithm,algorithmic}
\usepackage{amsmath, amsfonts}
\usepackage[listformat=empty]{caption}
\usepackage{amsmath,latexsym,amssymb,epsf,amsfonts,amsthm}
\usepackage[margin=3cm]{geometry}
\usepackage{epsf}
\usepackage{color}
\usepackage{mathrsfs}
\usepackage{epstopdf}
\usepackage{threeparttable}
\usepackage{graphicx}
\usepackage{subfig}
\usepackage{tabularx}
\usepackage[all]{xy}
\def \A {{\mathcal A}}
\def \P {{\mathcal P}}

\def \I {{\mathcal I}}

\def \R {{\mathbb R}}
\def \N {{\mathbb N}}
\def \C {{\mathbb C}}
\def \H {{\mathbb H}}
\def \hX {\hat{X}}
\def \tX {\tilde{X}}
\def \tx {\tilde{x}}
\def \tLam {\tilde{\Lambda}}

\newcommand{\be}{\begin{equation}}
\newcommand{\ee}{\end{equation}}
\newcommand{\ba}{\begin{eqnarray*}}
\newcommand{\ea}{\end{eqnarray*}}
\newcommand{\bi}{\begin{itemize}}
\newcommand{\ei}{\end{itemize}}

\newcommand{\del}{\mbox{$\delta$}}

\newcommand{\lam}{\mbox{$\lambda$}}
\newcommand{\Lam}{\mbox{$\Lambda$}}
\newcommand{\Sig}{\mbox{$\Sigma$}}

\newtheorem{thm}{Theorem}[section]

\newtheorem{prop}{Proposition}[section]
\newtheorem{lem}{Lemma}[section]
\newtheorem{cor}{Corollary}[section]
\newtheorem{rmk}{Remark}[section]

\newcommand{\comments}[1]{}
\newcommand{\hf}{\frac{1}{2}}
\newcommand{\diag}{\mathrm{diag}}
\newcommand{\Diag}{\mathrm{Diag}}
\newcommand{\rank}{\mathrm{rank}}
\newcommand{\Tr}{\mathrm{Tr}}
\newcommand{\la}{\langle}
\newcommand{\ra}{\rangle}

\begin{document}
\title{PhaseLiftOff: an Accurate and Stable Phase Retrieval Method Based on Difference of Trace and Frobenius Norms}
\author{Penghang Yin
\thanks {Department of Mathematics, UC Irvine, Irvine,
CA 92697, USA, (penghany@uci.edu).}
\and Jack Xin \thanks {Department of Mathematics, UC Irvine, Irvine,
CA 92697, USA, (jxin@math.uci.edu).} }


\date{}
\pagestyle{myheadings}
\maketitle

\begin{abstract}
Phase retrieval aims to recover a signal $x \in \C^{n}$ from its
amplitude measurements $|\la x, a_i \ra |^2$, $i=1,2,\cdots,m$,
where $a_i$'s are over-complete basis vectors, with $m$ at least $3n -2$ to ensure
a unique solution up to a constant phase factor. The quadratic measurement becomes linear
in terms of the rank-one matrix $X = x x^*$. Phase retrieval is then a rank-one
minimization problem subject to linear constraint for which a convex relaxation based on trace-norm minimization
(PhaseLift) has been extensively studied recently. At $m=O(n)$, PhaseLift recovers with high probability
the rank-one solution. In this paper, we present
a precise proxy of rank-one condition via the difference of trace and Frobenius norms which we call
PhaseLiftOff. The associated least squares minimization with this penalty as regularization
is equivalent to the rank-one least squares
problem under a mild condition on the measurement noise.
Stable recovery error estimates are valid at $m=O(n)$ with high probability.
Computation of PhaseLiftOff minimization is carried out by a convergent
difference of convex functions algorithm. In our numerical example, $a_i$'s are Gaussian distributed.
Numerical results show that PhaseLiftOff outperforms PhaseLift and its nonconvex variant (log-determinant regularization), and
successfully recovers signals near the theoretical lower limit on the number of measurements without the noise.
\end{abstract}

\section{Introduction.}
Phase retrieval has been a long standing problem in
imaging sciences such as
X-ray crystallography, electron microscopy, array imaging, optics, signal processing,
\cite{GerSax_72,Fienup_78,Fienup_82,Ha_93,Zou_97} among others.
It concerns with signal recovery when only the amplitude
measurements (say of its Fourier transform) are available.
Major recent advances have been made
for phase retrieval by formulating it as a matrix completion and rank one minimization problem (PhaseLift)
which is relaxed and solved as a convex trace (nuclear) norm minimization problem
under sufficient measurement conditions \cite{CSV,Chai_11,CL,CESV}; see also \cite{BCE,BBCE} for related work.
An alternative viable approach makes use of random masks in
measurements to achieve uniqueness of solution with high probability \cite{Fann_2012,Fann_2013}.

In this paper, we study a nonconvex Lipschitz continuous metric, the difference of trace and Frobenius norms,
and show that its minimization characterizes the rank one solution exactly and serves as
a new tool to solve the phase retrieval problem. We shall see that it is more accurate
than trace norm or the heuristic log-determinant \cite{Faz_02,Faz_03}, and performs the best
when the number of measurements approaches the theoretical lower limit \cite{BCE}. We shall call our method
PhaseLiftOff, where Off is short for subtracting off Frobenius norm from the trace norm
in PhaseLift \cite{CSV,CESV}.

The phase retrieval problem aims to reconstruct an unknown signal $\hat{x}\in\C^n$
satisfying $m$ quadratic constraints
$$
|\la a_{i},\hat{x}\ra|^2 = b_i, \quad i = 1,\dots, m,
$$
where the bracket is inner product, $a_i\in\C^n$ and $b_i\in\R$.
Letting $X = xx^*\in\C^{n\times n}$ be a rank-1 positive semidefinite matrix ($*$ is conjugate transpose),
one can recast quadratic measurements as linear ones about $X$:
$$
|\la a_i,x\ra|^2 = a_i^*Xa_i, \quad i = 1,\dots,m.
$$
Thus we can define a linear operator $\A$ uniquely determined by the measurement matrix $A = (a_1,\dots,a_m)\in\C^{n\times m}$:
\begin{equation*}
\begin{array}{lll}
\H^{n\times n} &\rightarrow &\R^m\\
X & \mapsto & \diag(A^*XA)
\end{array}
\end{equation*}
which maps Hermitian matrices into real-valued vectors. Denote $\hat{x}\hat{x}^*$ by $\hX$, and suppose $b = (b_1,\dots,b_m)^{\mathrm{T}} = \A(\hX)\in\R^m$ is the measurement vector. Then the phase retrieval becomes the feasibility problem, being equivalent to a rank minimization problem:
\begin{equation}\label{rkmin}
\begin{array}{ll}
\mbox{find} &  X\in\C^{n\times n} \\
\mbox{s.t.} & \A(X) = b \\
& X\succeq 0 \\
& \rank(X) = 1. \\
\end{array}
\Leftrightarrow
\begin{array}{ll}
\min_{X\in\C^{n\times n}} & \rank(X)  \\
\mbox{s.t.} & \A(X) = b \\
& X\succeq 0. \\
\end{array}
\end{equation}
To arrive at the original solution $\hat{x}$ to
the phase retrieval problem, one needs to factorize the solution $\hX$ of (\ref{rkmin}) as $\hat{x}\hat{x}^*$. It gives
$\hat{x}$ up to multiplication by a constant scalar with unit modulus (a constant phase factor),
because if $\hat{x}$ solves the phase retrieval problem, so does $c\hat{x}$, for any $c\in\C$ with $|c| = 1$.
At least $3n-2$ intensity measurements are necessary to guarantee uniqueness (up to a constant phase factor)
of the solution to (\ref{rkmin}) \cite{F}, whereas $4n-2$ generic measurements
suffice for uniqueness with probability one \cite{BCE}.

Instead of (\ref{rkmin}), Cand\`{e}s \textit{et al.} \cite{CESV,CSV} suggest solving
the convex PhaseLift problem, namely minimizing the trace norm as a convex surrogate for the rank functional:
$$
\min_{X\in\C^{n\times n}} \; \Tr(X)  \quad \mbox{s.t.} \quad \A(X) = b, \; X\succeq 0.
$$
It is shown in \cite{CL} that if each $a_i$ is Gaussian or uniformly sampled on the sphere,
then with high probability, $m=O(n)$ measurements are sufficient to recover the ground truth $\hX$ via PhaseLift.
For the noisy case, the following variant is considered in \cite{CL}:
$$
\min_{X\in\C^{n\times n}} \; \|\A(X) - b\|_1 \quad \mbox{s.t.} \quad X\succeq 0.
$$
In this case, $b = \A(\hX) + e$ is contaminated by the additive noise $e\in\R^m$. Similarly, $m = O(n)$ measurements guarantee stable recovery in the sense that
the solution $X^{\mathrm{opt}}$ satisfies $\|X^{\mathrm{opt}}-\hX\|_F=O(\frac{\|e\|_1}{m})$
with probability close to 1.
On the computational side, the regularized trace-norm minimization is considered in \cite{CESV,CSV}:
\begin{equation}\label{unl1}
\min_{X\in\C^{n\times n}} \; \hf\|\A(X)-b\|_2^2 +\lam\Tr(X) \quad \mbox{s.t.} \quad X\succeq 0.
\end{equation}
If there is no noise, a tiny value of $\lam$ would work well.
However, when the measurements are noisy,
determining $\lam$ requires extra work, such as employing the cross validation technique.

Besides PhaseLift and its nonconvex variant (log-determinant) proposed in \cite{CESV},
related formulations such as feasibility problem or weak PhaseLift \cite{DH} and PhaseCut \cite{phasecut}
also lead to phase retrieval solutions under certain measurement conditions. PhaseCut is a convex relaxation
where trace minimization is in the form $\min_{U} \, \Tr (U M)$, where $M$ (resp., $U$) is a known (resp., unknown) positive
semidefinite Hermitian matrix, and diag($U$) = 1. The exact recovery (tightness) conditions for
PhaseLift and PhaseCut are studied in \cite{phasecut} and references therein.

From the point of view of energy minimization, the phase retrieval problem is simply:
\be\label{energymin}
\min_{X\in\C^{n\times n}} \|\A(X)-b\|_2^2 \quad \mbox{s.t.} \quad X\succeq 0, \; \rank(X)=1.
\ee
This is a least squares-type model applicable to both noiseless and noisy cases.
Our main contribution in this work is to reformulate the phase retrieval problem (\ref{energymin}) as a nearly equivalent nonconvex optimization problem that can be efficiently solved by the so-called difference of convex functions algorithm (DCA).
Specifically, we propose to solve the following regularization problem:
\begin{equation}\label{unl1l2}
\min_{X\in\C^{n\times n}} \; \varphi(X) := \hf\|\A(X)-b\|_2^2 +\lam(\Tr(X)-\|X\|_F) \quad \mbox{s.t.} \quad X\succeq 0.
\end{equation}
Recently the authors of \cite{ELX,YLHX} have reported that minimizing
the difference of $\ell_1$ and $\ell_2$ norms would promote sparsity when recovering
a sparse vector from linear measurements. The $\ell_1-\ell_2$ minimization
is extremely favorable for the reconstruction of the 1-sparse vector $x$
because $\|x\|_1-\|x\|_2$ attains the possible minimum value zero at such $x$.
Note that when $X\succeq 0$, $\Tr(X)$ is nothing but the $\ell_1$ norm
of the vector $\sigma(X)$ formed by $X$'s singular values and $\|X\|_F$ the $\ell_2$ norm. Thus
(\ref{unl1l2}) is basically the counterpart of $\ell_1-\ell_2$ minimization with
nonnegativity constraint discussed in \cite{ELX}. Similarly, $\Tr(X)-\|X\|_F$ is minimized
when $\sigma(X)$ is 1-sparse or equivalently $\rank(X)=1$.
\medskip

The rest of the paper is organized as follows. After setting notations and giving
preliminaries in Section 2, we establish the equivalence
between (\ref{unl1l2}) and (\ref{energymin}) under mild conditions on $\lam$ and $\|e\|_2$ in Section \ref{theory}.
In particular, the equivalence holds in the absence of noise.
We will see that $\lam$ plays a very different role in (\ref{unl1l2}) from
that in (\ref{unl1}), as we have much more freedom to choose $\lam$ in (\ref{unl1l2}).
We then introduce the DCA method for solving (\ref{unl1l2}) and analyze its convergence
in Section \ref{alg}. The DCA calls for solving a sequence of convex subproblems which we carry out
with the alternating direction method of multipliers (ADMM). As an extension, we tailor our method to the task of retrieving real-valued or nonnegative signals.
In Section \ref{numerical}, we show numerical results
demonstrating the superiority of our method through examples where the columns of $A$
are sampled from Gaussian distribution.
The PhaseLiftOff problem (\ref{unl1l2}) with the $\Tr(X) - \|X\|_F$ regularization produces far
more accurate phase retrieval than either the trace norm or $\log ({\rm det} (X + \varepsilon I))$ ($\varepsilon > 0$).
We also observe that for a full interval of regularization parameters, the DCA produces robust solutions in the presence of noise.
The concluding remarks are given in Section 6.

\medskip

\section{Notations and Preliminaries.}
For any $X,Y\in\C^{n\times n}$, $\la X,Y\ra = \Tr(X^*Y)$ is the inner product for matrices, which is a generalization of that for vectors.
The Frobenius norm of $X$ is $\|X\|_F = \sqrt{\la X, X \ra}$, while $X\circ Y$ denotes the entry-wise product,
namely $(X\circ Y)_{ij} = X_{ij}Y_{ij}$, $\forall i,j$. $\diag(X)\in\C^n$ extracts the diagonal elements of $X$. The spectral norm of $X$ is
$\|X\|_2$, while the nuclear norm of $X$ is $\|X\|_*$. We have the following elementary inequalities:
$$
\|X\|_2\leq\|X\|_F\leq\sqrt{\rank(X)}\|X\|_2,
$$
and
$$
\|X\|_F\leq\|X\|_*\leq\sqrt{\rank(X)}\|X\|_F.
$$

For any vector $x\in\R^m$, $\|x\|_1$ and $\|x\|_2$ are
the $\ell_1$ norm and $\ell_2$ norm respectively, while $\Diag(x)\in\R^{m\times m}$ is
the diagonal matrix with $x$ on its diagonal.

 We assume that $m\geq n$ and that $A$ is of full rank unless otherwise stated, i.e. $\rank(A)=n$. Recall that $\A(X) := \diag(A^*XA)$ is a linear operator from $\H^{n\times n}$ to $\R^m$, then the adjoint operator $\A^*$ is defined as $\A^*(x) := A\Diag(x)A^*\in\H^{n\times n}$ for all $x\in\R^m$.
Furthermore, the norms of $\A$ and $\A^*$ are given by
$$
\|\A\| := \sup_{X\in\H^{n\times n}\setminus\{0\}}\frac{\|\A(X)\|_2}{\|X\|_F}, \quad \|\A^*\| := \sup_{x\in\R^m\setminus\{0\}}\frac{\|\A^*(x)\|_F}{\|x\|_2}.
$$
Since $(\H^{n\times n}, \la \cdot, \cdot\ra)$ and $(\R^m, \la \cdot, \cdot\ra)$ are both Hilbert spaces, we have
\be\label{operator}
\|\A^*\|^2 = \|\A\|^2 = \|\A\A^*\|.
\ee

The following lemma will be frequently used in the proofs.
\begin{lem}\label{fundmental}
Suppose $X, \; Y\in\C^{n\times n}$ and $X, \; Y \succeq0$, then
\bi
\item[\textbf{1.}] $\la X, Y\ra \geq0$.
\item[\textbf{2.}] $\la X,Y\ra = 0 \Leftrightarrow XY = 0$.
\item[\textbf{3.}] $\|\A(X)\|_2=0 \Leftrightarrow X=0$.
\ei
\end{lem}
\begin{proof}
\textbf{(1)} Suppose $Y = U\Sig U^*$ is the singular value decomposition (SVD), let $Y^{\hf} := U\Sig^{\hf}U^*\succeq0$, where the diagonal elements of $\Sig^{\hf}$
are square roots of the singular values. Then we have $Y = Y^{\hf}Y^{\hf}$ and
$$
 \la X,Y \ra = \Tr(X^*Y) = \Tr(XY) = \Tr(Y^{\hf}XY^{\hf})\geq0.
$$
The last inequality holds because $Y^{\hf}XY^{\hf}\succeq0$.

\textbf{(2)} $"\Rightarrow"$ Further assume $\Sig = \left(\begin{array}{cc}
                                    \Sigma_1 & 0 \\
                                    0 & 0
                                  \end{array}\right)$,
where $\Sigma_1\succ0$, and let $Z = U^*XU\succeq0$.
By (1), we have $\Tr(Y^{\hf}XY^{\hf}) = \la X,Y \ra = 0$, thus $Y^{\hf}XY^{\hf}=0$. So
$$
0 = \Sigma^{\hf}U^*XU\Sigma^{\hf} = \Sigma^{\hf}Z\Sigma^{\hf} = \left(\begin{array}{cc}
          \Sigma_1^{\hf} & 0 \\
                 0 & 0
          \end{array}\right)
          \left(\begin{array}{cc}
          Z_{11} & Z_{12} \\
          Z_{12}^* & Z_{22}
          \end{array}\right)
          \left(\begin{array}{cc}
          \Sigma_1^{\hf} & 0 \\
                 0 & 0
          \end{array}\right)
          =
                    \left(\begin{array}{cc}
          \Sigma_1^{\hf}Z_{11}\Sigma_1^{\hf} & 0 \\
                 0 & 0
          \end{array}\right),
$$
then we have $\Sigma_1^{\hf}Z_{11}\Sigma_1^{\hf} =0$  and $Z_{11}=0$.
Next we want to show $Z_{12}=0$. Suppose $Z_{12}\neq0$, let us consider $v_c = \left(\begin{array}{c}
          cZ_{12}w  \\
            w
          \end{array}\right)\in\C^n$, where $w$ is a fixed vector making $Z_{12}w$ nonzero and $c\in\R$. Then since $Z\succeq0$, we have
$$
0\leq v_c^*Zv_c = (cw^*Z_{12}^*, w^*)\left(\begin{array}{cc}
          0 & Z_{12} \\
          Z_{12}^* & Z_{22}
          \end{array}\right)
          \left(\begin{array}{c}
          cZ_{12}w  \\
            w
          \end{array}\right)
          =  2c|Z_{12}w|^2 + w^*Z_{22}w, \quad \forall c\in\R.
$$
In above inequality, letting $c\to-\infty$ leads to a contradiction. Therefore $Z_{12}=0$. A simple computation gives
$U^*XU\Sig = Z\Sig = 0$, and thus $XY = XU\Sig U^* = 0$.

$"\Leftarrow"$ If $XY=0$, then $\la X,Y \ra = \Tr(X^*Y) = \Tr(XY)=0$

\textbf{(3)} $"\Rightarrow"$ Let $X^{\hf}\succeq0$ such that $X^{\hf}X^{\hf}=X$. Then
$$
0 = \|\A(X)\|_2 = \|\diag(A^*X^{\hf}X^{\hf}A)\|_2.
$$
So $\diag(A^*X^{\hf}X^{\hf}A)=0$ and thus $0 = \Tr(A^*X^{\hf}X^{\hf}A) = \|A^*X^{\hf}\|_F^2$. This together with $\rank(A)=n\leq m$ implies
$X^{\hf} = 0$.

$"\Leftarrow"$ Trivial.
\end{proof}

{\bf Karush-Kuhn-Tucker conditions.} Let us consider a first-order stationary point $\tilde{X}$ of the minimization problem
$$
\min_{X\in\C^{n\times n}} \; f(X) \quad \mbox{s.t.} \quad X\succeq 0.
$$
Suppose $f$ is differentiable at $\tX$, then there exists $\tilde{\Lam}\in\C^{n\times n}$, such that the following Karush-Kuhn-Tucker (KKT) optimality conditions hold:
\bi
\item[$\bullet$] Stationarity: $\nabla f(\tX) = \tLam$.
\item[$\bullet$] Primal feasibility: $\tX\succeq 0$.
\item[$\bullet$] Dual feasibility: $\tLam\succeq0$.
\item[$\bullet$] Complementary slackness: $\tX\tLam = 0$.
\ei
In order to make better use of the last condition, by Lemma \ref{fundmental} (2), we can express it as
\bi
\item[$\bullet$] Complementary slackness: $\la \tX, \tLam \ra= 0$.
\ei

\section{Exact and Stable Recovery Theory.}\label{theory}
In this section, we present the PhaseLiftOff theory for exact and stable recovery of complex signals.

\subsection{Equivalence.} We first develop mild conditions that guarantee the full
equivalence between Phase Retrieval (\ref{energymin}) and PhaseLiftOff (\ref{unl1l2}).
\medskip

\begin{thm}\label{equivalence}
Let $\A$ be an arbitrary linear operator from $\H^{n\times n}$ to $\R^m$, and let $b = \A(\hX) + e$. If $\|b\|_2>\|e\|_2$ and $\lam>\frac{\|\A\|\|e\|_2}{\sqrt{2}-1}$, suppose $X^{\mathrm{opt}}$ is a solution (global minimizer) to (\ref{unl1l2}), then $\rank(X^{\mathrm{opt}})=1$. Moreover, minimization problems (\ref{energymin}) and (\ref{unl1l2}) are equivalent in the sense that they share the same set of solutions.
\end{thm}
\begin{proof}
Let $X^{\mathrm{opt}}$ be a solution to (\ref{unl1l2}). Since $\varphi(\hX) = \hf\|e\|_2^2<\hf\|b\|_2^2=\varphi(0)$, $X^{\mathrm{opt}}\neq0$. Suppose $\rank(X^{\mathrm{opt}})=r\geq1$, and let
\begin{equation*}
X^{\mathrm{opt}} = U\Sig U^* =  (U_1,U_2)\left(\begin{array}{cc}
                                    \Sigma_1 & 0 \\
                                    0 & 0
                                  \end{array}\right)
(U_1,U_2)^*
=U_1\Sigma_1 U_1^*
\end{equation*}
be the SVD, where $U_1 = (u_1,\dots,u_r)\in \C^{n\times r}$, $U_2=(u_{r+1},\dots,u_n)\in \C^{n\times (n-r)}$, and $\Sigma_1= \Diag((\sigma_1,\dots,\sigma_r))\in \R^{r\times r}$ with $X^{\mathrm{opt}}$'s positive singular values on its diagonal.

Since $X^{\mathrm{opt}}$ is a global minimizer, it is also a stationary point. This means KKT conditions must hold at $X^{\mathrm{opt}}$, i.e., there exists $\Lam \in \C^{n\times n}$ such that
\begin{align}\label{opt}
& \A^*(\A(X^{\mathrm{opt}})-b) + \lam(I_n - \frac{X^{\mathrm{opt}}}{\|X^{\mathrm{opt}}\|_F}) = \Lam,\\
& X^{\mathrm{opt}}\succeq 0, \; \Lam\succeq0, \; \la X^{\mathrm{opt}}, \Lam \ra = 0.\notag
\end{align}
Rewrite $I_n = UU^* = U_1U_1^* + U_2U_2^*$, then (\ref{opt}) becomes
\begin{align*}
-\A^*(\A(X^{\mathrm{opt}})-b)  = \lam(I_n - \frac{X^{\mathrm{opt}}}{\|X^{\mathrm{opt}}\|_F}) - \Lam = \lam U_1U_1^* + (\lam U_2U_2^*-\Lam)-\lam \frac{X^{\mathrm{opt}}}{\|X^{\mathrm{opt}}\|_F}.
\end{align*}
Taking Frobenius norm of both sides above, we obtain
\begin{align}\label{ineq1}
\|\A^*(\A(X^{\mathrm{opt}})-b)\|_F & = \| \lam U_1U_1^* + (\lam U_2U_2^*-\Lam)-\lam \frac{X^{\mathrm{opt}}}{\|X^{\mathrm{opt}}\|_F} \|_F
\geq \| \lam U_1U_1^* + (\lam U_2U_2^*-\Lam) \|_F -  \lam.
\end{align}
Also, we have $0 = \la X^{\mathrm{opt}},\Lam\ra = \la U_1\Sigma_1 U_1^*, \Lam\ra = \sum_{i=1}^r\sigma_i\la u_iu_i^*,\Lam \ra$. But $\la u_iu_i^*,\Lam \ra\geq0$, since $\Lam\succeq0$ and $u_iu_i^*\succeq0$. So $\la u_i u_i^*, \Lam\ra = 0$ for $1\leq i\leq m$, and $\la U_1 U_1^*, \Lam\ra = \sum_{i=1}^r\la u_iu_i^*,\Lam \ra
= 0$. Moreover,
$$
\la U_1 U_1^*, U_2U_2^*\ra = \sum_{i=1}^r\sum_{j=r+1}^n \la u_iu_i^*,u_ju_j^* \ra = \sum_{i=1}^r\sum_{j=r+1}^n \la u_j^*u_i,u_j^*u_i \ra = 0.
$$
In a word, $U_1U_1^*$ is orthogonal to both $U_2U_2^*$ and $\Lam$. Then from Pythagorean theorem it follows that
$$
\| \lam U_1U_1^* + (\lam U_2U_2^*-\Lam) \|_F = \sqrt{\|\lam U_1U_1^*\|_F^2 + \|(\lam U_2U_2^*-\Lam) \|_F^2}\geq \lam\| U_1U_1^*\|_F,
$$
and thus (\ref{ineq1}) reduces to
\be\label{ineq2}
\|\A^*(\A(X^{\mathrm{opt}})-b)\|_F \geq \lam\| U_1U_1^*\|_F  -  \lam = \lam(\sqrt{r}-1).
\ee
On the other hand, since
$$
\|\A(X^{\mathrm{opt}})-b\|_2 \leq \sqrt{\|\A(X^{\mathrm{opt}})-b\|_2^2+2\lam(\Tr(X^{\mathrm{opt}})-\|X^{\mathrm{opt}}\|_F)} = \sqrt{2\varphi(X^{\mathrm{opt}}}),
$$
we have
\begin{align}\label{ineq3}
\|\A^*(\A(X^{\mathrm{opt}})-b)\|_F & \leq \|\A^*\| \|\A(X^{\mathrm{opt}})-b\|_2 = \|\A\| \|\A(X^{\mathrm{opt}})-b\|_2 \notag\\
& \leq \|\A\|\sqrt{2\varphi(X^{\mathrm{opt}})}\leq\|\A\|\sqrt{2\varphi(\hX)}=\|\A\|\|e\|_2.
\end{align}
Combining (\ref{ineq2}) and (\ref{ineq3}) gives $\lam(\sqrt{r}-1)\leq \|\A\|\|e\|_2$, or equivalently
$$
r\leq (\frac{\|\A\|\|e\|_2}{\lam}+1)^2 <2.
$$
The last inequality above follows from the assumption $\lam>\frac{\|\A\|\|e\|_2}{\sqrt{2}-1}$. $r$ is a natural number, so $r = 1$.

Note that $\Tr(X)-\|X\|_F \geq0$ for $X\succeq0$ with equality when $\rank(X)=1$. It is not hard to see the equivalence between (\ref{energymin}) and (\ref{unl1l2}).
\end{proof}

\begin{cor}\label{noiseless}
In the absence of measurement noise, the equivalence
between (\ref{energymin}) and (\ref{unl1l2}) holds for all $\lam>0$.
In this sense, (\ref{unl1l2}) is essentially a parameter-free model.
\end{cor}

Theorem \ref{equivalence} claims that provided the noise in measurement is smaller than the measurement itself, all $\lam$ that exceed an explicit threshold would work equally well for (\ref{unl1l2}) in theory. In contrast, the $\lam$ in (\ref{unl1}) needs to be carefully chosen to balance the fidelity and penalty terms. Particularly in noiseless case, the $\lam$ in (\ref{unl1l2}) acts like a 'fool-proof' regularization parameter, and $\A(X) = b$ is always exact at the solution $X^{\mathrm{opt}}=\hX$ whenever $\lam>0$, whereas a perfect reconstruction via solving (\ref{unl1}) generally requires a dynamic $\lam$ that goes to 0.

\begin{rmk}
Despite the tremendous room for $\lam$ values in view of Theorem \ref{equivalence}, we should point out that in practice the choice of $\lam$ could be more subtle because
\bi
\item[$\bullet$] The theoretical lower bound $\frac{\|\A\|\|e\|_2}{\sqrt{2}-1}$ for $\lam$ may be too stringent, and a smaller $\lam$ could also be feasible.
\item[$\bullet$] Choosing $\lam$ too large may reduce the mobility of the energy minimizing iterations due to trapping by local minima.
\ei
An efficient algorithm designed for PhaseLiftOff should be as insensitive as possible to the choice of $\lam$ when it is large enough.
\end{rmk}

\subsection{Exact and stable recovery under Gaussian measurements.}
In the framework of \cite{CSV, CL}, assuming $a_i$'s are i.i.d. complex-valued normally distributed random vectors,
we establish the exact recovery and stability results for (\ref{unl1l2}).
Due to the equivalence between (\ref{energymin}) and (\ref{unl1l2}) under the
conditions stated in Theorem \ref{equivalence}, it suffices to discuss
the model (\ref{energymin}) only. Similar to \cite{CL}, $m = O(n)$ measurements suffice
to ensure exact recovery in noiseless case or stability in noisy case with probability close to 1.
Although the required number of measurements for (\ref{energymin}) and that for PhaseLift are both on the minimal order $O(n)$,
the scalar factor of the former is actually smaller, and so is the probability of failure.
\medskip

\begin{thm}\label{stability}
Suppose column vectors of $A$ are i.i.d. complex-valued normally distributed. Fix $\alpha\in(0,1)$, there are constants $\theta, \gamma>0$ such that if $m > \theta[\alpha^{-2}\log \alpha^{-1}]n$, for any $\hX$, (\ref{energymin}) is stable in the sense that its solution $X^{\mathrm{opt}}$ satisfies
\be\label{errbd}
\|X^{\mathrm{opt}}- \hX\|_F \leq C_\alpha\frac{\|e\|_2}{\sqrt{m}}
\ee
for some constant $C_\alpha :=  \frac{\sqrt{2}}{(\sqrt{2}-1)(1-\alpha)} > 0$ with probability at least $1-3e^{-\gamma m \alpha^2}$. In particular, when $e=0$, the recovery is exact.
\end{thm}
The proof is straightforward with the aid of Lemma 5.1 in \cite{CSV}:
\begin{lem}[\cite{CSV}]\label{lemma}
Under the assumption of Theorem \ref{stability}, we have that $\A$ obeys the following property with
probability at least $1-3e^{-\gamma m \alpha^2}$: for any Hermitian matrix $X$ with $\rank(X)\leq 2$,
$$
\frac{1}{m}\|\A(X)\|_1 \geq 2(\sqrt{2}-1)(1-\alpha)\|X\|_2.
$$
\end{lem}
{\bf Proof of Theorem \ref{stability}.}
\begin{proof}
Let $X^{\mathrm{opt}} = \hX + H$, where $\hX $ satisfies $\A(\hX) + e = b$,
then $H$ is Hermitian with $\rank(H)\leq 2$. Since
$$
\|e\|_2 = \|\A(\hX) - b \|_2 \geq \|\A(X^{\mathrm{opt}}) - b\|_2 \geq \|\A(X^{\mathrm{opt}}-\hX)\|_2-\|\A(\hX) - b\|_2,
$$
we have $\|\A(H)\|_2 \leq 2\|e\|_2$. Invoking Lemma \ref{lemma} above, we further have
$$
{1\over \sqrt{m}} \|\A(H)\|_2 \geq \frac{1}{m}\|\A(H)\|_1 \geq 2(\sqrt{2}-1)(1-\alpha)\|H\|_2 \geq \frac{2(\sqrt{2}-1)(1-\alpha)}{\sqrt{2}}\|H\|_F.
$$
Therefore,
$$
\|X^{\mathrm{opt}}- \hX\|_F = \|H\|_F\leq \frac{\sqrt{2}}{(\sqrt{2}-1)(1-\alpha)}\frac{\|e\|_2}{\sqrt{m}}.
$$
The above inequality holds with probability at least $1-3e^{-\gamma m \alpha^2}$.
\end{proof}

\subsection{Computation of $\|\A\|$.}
The $\|\A\|$ in Theorem \ref{equivalence} can be actually computed. To do this, we first prove the following result:
\begin{lem}\label{inv}
$\A\A^*(x) = (A^*A\circ\overline{A^*A})x$, $\forall x\in\R^m$, where the overline denotes complex conjugate.
\end{lem}
\begin{proof}
By the definitions of $\A$ and $\A^*$, $\A\A^*(x)= \diag(A^*A\Diag(x)A^*A)$, then $\forall 1\leq i\leq m$, the $i$-th entry of $\A\A^*(x)$ reads
\begin{align*}
(\A\A^*(x))_i & = (A^*A\Diag(x)A^*A)_{ii} = \sum_{j=1}^m x_j(A^*A)_{ij}(A^*A)_{ji} \\
& = \sum_{j=1}^m x_j(A^*A)_{ij}\overline{(A^*A)_{ij}} = \sum_{j=1}^m x_j(A^*A\circ\overline{A^*A})_{ij}\\
& = ((A^*A\circ\overline{A^*A})x)_i.
\end{align*}
\end{proof}
Hence, from Lemma \ref{inv} and (\ref{operator}) it follows that
$$
\|\A\| = \sqrt{\|\A\A^*\|} = \sqrt{\|A^*A\circ\overline{A^*A}\|_2}.
$$
It would be interesting to see how fast $\|\A\|$ grows with dimensions $n$ and $m$ when $A$ is a complex-valued random Gaussian matrix. In this setting, $\A$ enjoys approximate $\ell_1$-isometry properties as revealed by Lemma 3.1 of \cite{CSV} (in complex case). Here we are most interested in the part that concerns the upper bound:
\begin{lem}[\cite{CSV}]\label{isometry}
Suppose $A\in\C^{n\times n}$ is random Gaussian. Fix any $\delta>0$ and assume $m\geq16\delta^{-2}n$. Then with probability at least $1-e^{-m\epsilon^2/2}$, where $\delta/4 = \epsilon^2 + \epsilon$,
$$
\frac{1}{m}\|\A(X)\|_1\leq(1+\delta)\|X\|_*
$$
holds for all $X\in\C^{n\times n}$.
\end{lem}

Under assumptions of Lemma \ref{isometry}, with high probability we have
$$
\frac{1}{\sqrt{m}}\|\A(X)\|_2\leq\frac{1}{m}\|\A(X)\|_1\leq(1+\delta)\|X\|_*\leq(1+\delta)\sqrt{n}\|X\|_F,
$$
which implies $\|\A\| = O(\sqrt{mn})$. For the phase retrieval problem to be well-posed, $m=O(n)$ is required; for instance, $m = 4n$ would be sufficient according to \cite{BCE}. Then we expect that $\|\A\|$ is on the order of $n$. This can be validated by a simple numerical experiment whose results are shown in Table \ref{Anorm} below.

\begin{table}[htbp]
\centering
\begin{tabular}{|p{30pt}|p{30pt}p{30pt}p{30pt}p{30pt}p{30pt}p{30pt}|}
\toprule
$n$ & 32 & 64 & 128 & 256 & 512 & 1024\\
\midrule
$\|\A\|$ & 148 & 291 & 577 & 1149 & 2295 & 4584\\
\bottomrule
\end{tabular}
\caption{Fixing $m = 4n$, $\|\A\|$ is nearly linear in $n$, where $\|\A\| = \sqrt{\|A^*A\circ\overline{A^*A}\|_2}$ with $A$ being complex-valued Gaussian matrix. For each $n$, the value of $\|\A\|$ is averaged over 10 independent samples of $A$ using MATLAB.}\label{Anorm}
\end{table}

\section{Algorithms.}\label{alg}
In this section, we consider the computational aspects of the minimization problem (\ref{unl1l2}).
\subsection{Difference of convex functions algorithm.}
The DCA is a descent method without line search developed by Tao and An \cite{TA_97, TA_98}.
It addresses the problem of minimizing a function of the form $f(x) = g(x) - h(x)$ on the space $\R^n$, with $g$, $h$ being lower semicontinuous proper convex functions:
$$
\min_{x\in\R^n} \; f(x)
$$
$g-h$ is called a DC decomposition of $f$, while the convex functions $g$ and $h$ are DC components of $f$.
The DCA involves the construction of two sequences $\{x^k\}$ and $\{y^k\}$, the candidates for optimal solutions of primal
and dual programs respectively.
At the $(k+1)$-th step, we choose a subgradient of $h(x)$ at $x^k$, namely $y^k \in \partial h(x^k)$. We then linearize $h$ at $x^k$,
which permits a convex upper envelope of $f$. More precisely,
$$
f(x) = g(x)-h(x) \leq g(x) - (h(x^k) + \langle y^k, x-x^k \rangle), \; \forall x\in\R^n
$$
with equality at $x = x^k$.

By iteratively computing
\begin{equation*}
\begin{cases}
y^k \in \partial h(x^k),\\
x^{k+1} = \arg\min_{x\in\R^n} g(x) - (h(x^k) + \langle y^k, x-x^k \rangle)
\end{cases}
\end{equation*}
we have
$$
f(x^k) \geq g(x^{k+1}) - (h(x^k) + \langle y^k, x^{k+1}-x^k \rangle)\geq g(x^{k+1})-h(x^{k+1})= f(x^{k+1}).
$$
This generates a monotonically decreasing sequence $\{f(x^k)\}$, leading to its convergence if $f(x)$ is bounded from below.

We can readily apply the DCA to (\ref{unl1l2}), where the objective naturally has the DC decomposition
\be\label{dcdecomp}
\varphi(X) = (\frac{1}{2}\|\A(X) - b\|_2^2 + \lam\Tr(X)) - \lam\|X\|_F.
\ee
Since $\varphi(X)\geq0$ for all $X\succeq0$, the scheme
\begin{equation*}
\begin{cases}
\Delta^k \in \partial \|X^k\|_F,\\
X^{k+1} = \arg\min_{X\in\C^{n\times n}} \; \frac{1}{2}\|\A(X) - b\|_2^2 + \lam\Tr(X)- \lam(\|X^k\|_F + \la \Delta^k, X-X^k \ra) \quad \mbox{s.t.} \quad X\succeq0.
\end{cases}
\end{equation*}
yields a decreasing and convergent sequence $\{\varphi(X^k)\}$.
Note that $\|X\|_F$ is differentiable with gradient $\frac{X}{\|X\|_F}$ at all $X\neq0$ and that $0 \in \partial \|X\|_F$ at $X=0$,
by ignoring constants we iterate
\begin{equation}\label{iter}
X^{k+1} =
\begin{cases}
\arg\min_{X\in\C^{n\times n}} \; \frac{1}{2}\|\A(X) - b\|_2^2 + \lam\Tr(X) \quad \mbox{s.t.} \quad X\succeq0 & \mbox{if}\quad X^k = 0,\\
\arg\min_{X\in\C^{n\times n}} \; \frac{1}{2}\|\A(X) - b\|_2^2 + \lam\langle X, I_n - \frac{X^k}{\|X^k\|_F}\rangle \quad \mbox{s.t.} \quad X\succeq0 & \mbox{otherwise}.\\
\end{cases}
\end{equation}
Since $X^k-X^{k-1}\to0$ as $k\to\infty$ (Proposition \ref{results} (2)), we stop the DCA when
$$
\frac{\|X^k-X^{k-1}\|_F}{\max\{\|X^k\|_F,1\}}<\verb|tol|,
$$
for some given tolerance $\verb|tol|>0$. In practice the above iteration takes only a few steps to convergence.
While the problem (\ref{unl1l2}) is nonconvex, empirical studies have
shown that the DCA usually produces a global minimizer with a good initialization.
In particular, our initialization here is $X^0 = 0$, as suggested by the observations in \cite{YLHX}.
This amounts to employing the (PhaseLift) solution of the regularized
trace-norm minimization problem (\ref{unl1}) as a start.

\subsection{Convergence analysis.}\label{convergence}
We proceed to show that the sequence $\{X^k\}$ is bounded and $X^{k+1}-X^k \to 0$, and limit points of $\{X^
k\}$ are stationary points of (\ref{unl1l2}) satisfying KKT optimality conditions.
Standard convergence results for the general DCA (e.g. Theorem 3.7 of \cite{TA_98}) take advantage of
strong convexity of the DC components. However, the DC components in (\ref{dcdecomp}) only
possess weak convexity as $\ker(\A^{*}\A)$ is generally nontrivial. In this sense, our analysis below is novel.

\begin{lem}\label{coercive}
Suppose $X\succeq0$, $\varphi(X) \to \infty$ as $X\to\infty$.
\end{lem}
\begin{proof}
It suffices to show that for any fixed nonzero $X\succeq0$, $\varphi(cX)\to\infty$ as $c\to\infty$.
\begin{align*}
\varphi(cX) =  \frac{1}{2}\|c\A(X) - b\|_2^2 +c\lam(\Tr(X) - \|X\|_F)\geq  \frac{1}{2}(c\|\A(X)\|_2-\|b\|_2)^2.
\end{align*}
Since $X\succeq0$ and is nonzero, by Lemma \ref{fundmental} (3), $\|\A(X)\|_2 > 0$. Hence, $c\|\A(X)\|_2-\|b\|_2\to\infty$ as $c\to \infty$, which completes the proof.
\end{proof}

\begin{lem}\label{decrease}
Let $\{X^k\}$ be the sequence generated by the DCA. For all $k\in\N$, we have
 \be\label{lbd}
 \varphi(X^k) - \varphi(X^{k+1})\geq \hf\|\A(X^k-X^{k+1})\|_2^2 + \lam(\|X^{k+1}\|_F- \|X^k\|_F - \la \Delta^k, X^{k+1}-X^k \ra)\geq0,
 \ee
where $\Delta^{k}\in\partial\|X^{k}\|_F$.
\end{lem}
\begin{proof}
We first calculate
\begin{align}\label{k+1_k}
\varphi(X^k) - \varphi(X^{k+1}) = & \hf\|\A(X^k-X^{k+1})\|_2^2 + \la \A(X^k - X^{k+1}),\A(X^{k+1})-b \ra  \notag\\
& + \lam\Tr(X^k-X^{k+1}) + \lam(\|X^{k+1}\|_F-\|X^{k}\|_F).
\end{align}
Recall that the $(k+1)$-th DCA iteration is to solve
$$
X^{k+1} = \arg\min_{X\in\C^{n\times n}} \; \frac{1}{2}\|\A(X) - b\|_2^2 + \lam\langle X, I_n - \Delta^{k}\rangle \quad \mbox{s.t.} \quad X\succeq 0,
$$
where $\Delta^{k}\in\partial\|X^{k}\|_F$. Then by the KKT conditions at $X^{k+1}$, there exists $\Lam^{k+1}$ such that
\begin{eqnarray}\label{optk}
 & \A^*(\A(X^{k+1})-b)+\lam (I_n  - \Delta^{k}) = \Lam^{k+1},\\
 & X^{k+1}\succeq0, \; \Lam^{k+1}\succeq0, \; \la \Lam^{k+1},X^{k+1} \ra = 0.\notag
\end{eqnarray}
Multiplying (\ref{optk}) by $X^k-X^{k+1}$ (inner product) gives
\be\label{ineq4}
\la \A(X^k - X^{k+1}),\A(X^{k+1})-b \ra + \lam\Tr(X^k-X^{k+1})  = \la \Lam^{k+1}, X^{k}\ra - \lam \la \Delta^k, X^{k+1}-X^{k} \ra.
\ee
In (\ref{ineq4}), $\la \Delta^k, X^{k+1}-X^{k} \ra \leq \|X^{k+1}\|_F - \|X^{k}\|_F$ since $\Delta^{k}\in\partial\|X^{k}\|_F$, and
$\la \Lam^{k+1}, X^{k}\ra \geq0$ by Lemma \ref{fundmental} (1). Combining (\ref{k+1_k}) and (\ref{ineq4}) gives
\begin{align*}
\varphi(X^k) - \varphi(X^{k+1}) = & \hf\|\A(X^k-X^{k+1})\|_2^2+ \lam(\|X^{k+1}\|_F-\|X^{k}\|_F)+ \la \Lam^{k+1}, X^{k}\ra \\
& - \lam \la \Delta^k, X^{k+1}-X^{k}\ra\\
\geq & \hf\|\A(X^k-X^{k+1})\|_2^2 + \lam(\|X^{k+1}\|_F- \|X^k\|_F - \la \Delta^k,X^{k+1}-X^k\ra)\\
\geq & 0.
\end{align*}
\end{proof}

We are now in the position to prove convergence results of the DCA for solving the PhaseLiftOff problem (\ref{unl1l2}).
\begin{prop}\label{results}
Let $\{X^k\}$ be the sequence produced by the DCA starting with $X^0 = 0$.
\bi
\item[\textbf{1.}] $\{X^k\}$ is bounded.
\item[\textbf{2.}] $X^{k+1}-X^k\to0$ as $k\to\infty$.
\item[\textbf{3.}] Any nonzero limit point $\tilde{X}$ of the sequence $\{X^k\}$ is a first-order stationary point, which means there exists $\tilde{\Lam}$, such that the following KKT conditions are satisfied:
    \begin{itemize}
    \item[$\bullet$] Stationarity: $\A^*(\A(\tilde{X})-b) + \lam(I_n - \frac{\tilde{X}}{\|\tilde{X}\|_F}) = \tilde{\Lam}$.
    \item[$\bullet$] Primal feasibility: $\tilde{X}\succeq 0$.
    \item[$\bullet$] Dual feasibility: $\tilde{\Lam}\succeq0$.
    \item[$\bullet$] Complementary slackness: $\la \tilde{X}, \tilde{\Lam} \ra = 0$.
    \end{itemize}
\ei
\end{prop}

\begin{proof}
\textbf{(1)} By Lemma \ref{coercive}, the level set $\Omega := \{X\in \C^{n\times n}: X\succeq0, \varphi(X)\leq \varphi(0)\}$ is bounded. Since $\{\varphi(X^k)\}$ is decreasing, $\{X^k\}\subseteq\Omega$ is also bounded.

\textbf{(2)} Letting $k=0$ and substituting $\Delta^0 = 0$ in (\ref{lbd}), we obtain
$$
\varphi(0) - \varphi(X^{1})\geq \hf\|\A(X^1)\|_2^2 + \lam\|X^1\|_F.
$$
If $X^1\neq0$, then $\varphi(0) > \varphi(X^{1})\geq\cdots\geq\varphi(X^{k})$, so $X^k \neq 0$, $\forall k\geq1$. Otherwise $X^k \equiv 0$.

Assuming $X^k \neq 0$, we show that $X^{k+1}-X^k\to0$ as $k\to\infty$ in what follows.
Note that $\{\varphi(X^k)\}$ is decreasing and convergent, and that $\Delta^k = \frac{X^k}{\|X^{k}\|_F}$ when $k\geq1$. Combining this with (\ref{lbd}), we have the following key information about $\{X^k\}$:
\begin{align}
& \|\A(X^k-X^{k+1})\|_2\to0\label{lim1}\\
& \|X^{k+1}\|_F - \la \frac{X^k}{\|X^{k}\|_F},X^{k+1} \ra\to 0. \label{lim2}
\end{align}
Define $c^k := \frac{\la X^k,X^{k+1}\ra}{\|X^k\|_F^2}\geq0$ and $E^k := X^{k+1}-c^kX^k$, then it suffices to prove $E^k\to0$ and $c^k\to1$. A simple computation shows
\begin{align*}
\|E^k\|_F^2 = \|X^{k+1}\|_F^2 - \frac{\la X^k,X^{k+1} \ra^2}{\|X^k\|_F^2}\to 0,
\end{align*}
where ($\ref{lim2}$) was used. Thus, from (\ref{lim1}) it follows that
$$
0 = \lim_{k\to\infty}\|\A(X^{k}-X^{k+1})\|_2 = \lim_{k\to\infty}\|\A((c^k-1)X^k - E^k)\|_2 = \lim_{k\to\infty}|c^k-1|\|\A(X^k)\|_2.
$$
Suppose $\lim_{k\to\infty}c^k\neq1$, then there exists a subsequence $\{X^{k_j}\}$ such that $\|\A(X^{k_j})\|_2\to0$. Since, by Lemma \ref{fundmental} (3), $\A(X)=0\Leftrightarrow X=0$ for $X\succeq0$, we must have $X^{k_j}\to0$ and $\varphi(X^{k_j})\to\varphi(0)$, which leads to a
contradiction because
$$
\varphi(X^{k_j})\leq\varphi(X^1)<\varphi(0).
$$
Therefore $c^k\to1$ and $X^{k+1}-X^{k}\to0$, as $k\to\infty$.

\textbf{(3)} Let $\{X^{k_j}\}$ be a subsequence of $\{X^{k}\}$ converging to some limit point $\tilde{X}\neq0$, then the optimality conditions at the $k_j$-th step read:
\begin{align*}
& \A^*(\A(X^{k_j})-b)+\lam (I_n  - \frac{X^{k_j-1}}{\|X^{k_j-1}\|_F}) = \Lam^{k_j}, \\
& X^{k_j}\succeq0, \;  \Lam^{k_j}\succeq0, \; \la \Lam^{k_j},X^{k_j} \ra = 0.
\end{align*}
Define
\begin{align*}
\tilde{\Lam}:= & \lim_{k_j\to\infty} \Lam^{k_j} \\
= & \lim_{k_j\to\infty} \A^*(\A(X^{k_j})-b)+\lam (I_n  - \frac{X^{k_j-1}}{\|X^{k_j-1}\|_F}) \\
= & \lim_{k_j\to\infty} \A^*(\A(X^{k_j})-b)+\lam (I_n  - \frac{X^{k_j}}{\|X^{k_j}\|_F}) + \lam(\frac{X^{k_j}}{\|X^{k_j}\|_F}-\frac{X^{k_j-1}}{\|X^{k_j-1}\|_F}) \\
= & \A^*(\A(\tilde{X})-b)+\lam (I_n  - \frac{\tilde{X}}{\|\tilde{X}\|_F}).
\end{align*}
In the last equality, we used $\lim_{k_j\to\infty}X^{k_j} = \tilde{X}\neq0$ and $X^{k_j}-X^{k_j-1}\to0$. Since $X^{k_j}\succeq0$, $\Lam^{k_j}\succeq0$, their limits are $\tilde{X}\succeq0$ and $\tilde{\Lam}\succeq0$. It remains to check that $\la \tilde{\Lam},\tilde{X} \ra = 0$.
Using $\la \Lam^{k_j},X^{k_j} \ra = 0$, we have
\begin{align*}
\la \tilde{\Lam},\tilde{X} \ra= \la \tilde{\Lam}-\Lam^{k_j},\tilde{X} - X^{k_j}\ra + \la \Lam^{k_j},\tilde{X}\ra + \la \tilde{\Lam},X^{k_j}\ra.
\end{align*}
Let $k_j\to\infty$ on the right hand side above, $\la \tilde{\Lam},\tilde{X} \ra = 0$.
\end{proof}

\begin{rmk}
 In light of the proof of Proposition \ref{results} (2), one can see that for all $k\geq1$, either $X^k\equiv0$ or $\|X^k\|_F>\eta$ for some $\eta>0$. A sufficient condition to ensure that the DCA does not yield $\tilde{X}=0$ is as follows
$$
\hf\|b\|_2^2 > \hf\|e\|_2^2 + \lam\Tr(\hat{X}) \Leftrightarrow \lam < \frac{\|b\|_2^2-\|e\|_2^2}{2\Tr(\hX)},
$$
where $\hX$ is the ground truth obeying $b = \A(\hX)+e$. The above condition would guarantee that $X^1\neq0$.
Though the equivalence between (\ref{energymin})
and (\ref{unl1l2}) follows from Theorem \ref{equivalence} as long as $\lam$ is sufficiently large,
in practice $\lam$ cannot get too large because the DCA iterations may stall at $X^0=0$.
\end{rmk}

\subsection{Solving the subproblem.}
At the $(k+1)$-th DC iteration, one needs to solve a convex subproblem of the form:
\begin{equation}\label{SDP}
X^{k+1} = \arg\min_{X\in\C^{n\times n}} \; \hf\|\A(X) - b\|_2^2 + \la X,W\ra \quad \mbox{s.t.} \quad X \succeq 0.
\end{equation}
In our case, $W = \lam I_n$ or $\lam(I_n - \frac{X^k}{\|X^k\|_F})$ is a known Hermitian matrix.

This problem can be treated as a weighted trace-norm regularization problem, which has been
studied in \cite{CESV}. The authors of \cite{CESV} suggest using FISTA \cite{FISTA,CESV} which is a variant of Nesterov's accelerated
gradient descent method \cite{Nesterov}. An alternative choice is the alternating direction method of multipliers (ADMM).
We only discuss the naive ADMM here, though this algorithm could be further accelerated
by incorporating Nesterov's idea \cite{fastADMM}.
To implement ADMM, we introduce a dual variable $Y$ and form the augmented Lagrangian
\begin{equation}\label{lag}
\mathcal{L}_{\del}(X,Y,Z) = \hf\|\A(X) - b\|_2^2 + \la X,W\ra + \la Y,X-Z \ra + \frac{\del}{2}\|X-Z\|_F^2 + g_{\succeq}(Z),
\end{equation}
where
\begin{equation*}
g_{\succeq}(Z) =
\begin{cases}
0 & \mbox{if} \quad Z\succeq0,\\
\infty & \mbox{otherwise}.
\end{cases}
\end{equation*}
ADMM consists of updates on both the primal and dual variables \cite{ADMM}:
\begin{equation*}
\begin{cases}
 X^{l+1} =  \arg\min_X \; \mathcal{L}_{\del}(X,Y^l,Z^l)\\
 Z^{l+1} =  \arg\min_Z \; \mathcal{L}_{\del}(X^{l+1},Y^l,Z)\\
 Y^{l+1} = Y^l + \del(X^{l+1} - Z^{l+1})
\end{cases}
\end{equation*}
The first two steps have closed-form solutions, which are detailed in Algorithm \ref{ADMM}.
\begin{algorithm}
\caption{ADMM for solving (\ref{SDP})}
\begin{algorithmic}\label{ADMM}
\WHILE {not converged}
\STATE $X^{l+1} = (\A^*\A + \del \I_n)^{-1}(\A^*(b)- W + \del Z^l -Y^l)$
\STATE $Z^{l+1} = \P_{\succeq}(X^{l+1} + Y^l/\del)$
\STATE $Y^{l+1} = Y^l + \del(X^{l+1} - Z^{l+1})$
\ENDWHILE
\end{algorithmic}
\end{algorithm}
In the $X$-update step, one needs to know the expression of $(\A^*\A + \del \I_n)^{-1}$.
The celebrated Woodbury formula implies
$$
(\A^*\A + \del \I_n)^{-1} = \frac{1}{\del}(\I_n - \A^*(\A\A^*+\del\I_m)^{-1}\A).
$$
By Lemma \ref{inv}, $\A\A^* = A^*A\circ\overline{A^*A}$, so we have
\begin{align*}
(\A^*\A + \del \I_n)^{-1}(X) = & \frac{1}{\del}(X - \A^*((\A\A^*+\del\I_m)^{-1}\A(X))\\
= & \frac{1}{\del}(X - A\Diag((A^*A\circ\overline{A^*A} + \del I_m)^{-1}\diag(A^*XA))A^*)
\end{align*}
In the $Z$-update step, $\P_{\succeq}:\H^{n\times n} \rightarrow \H^{n\times n}$ represents the projection onto the positive semidefinite cone. More precisely, if $X$ has the eigenvalue decomposition
$X = U\Sig U^*$, then
$$
\P_{\succeq}(X) = U\max\{\Sig,0\}U^*.
$$
According to \cite{ADMM}, the stopping criterion here is given by:
$$
\|R^l\|_F\leq n\epsilon^{\mathrm{abs}}+\epsilon^{\mathrm{rel}}\max\{\|X^l\|_F,\|Z^l\|_F\}, \quad
\|S^l\|_F\leq n\epsilon^{\mathrm{abs}}+\epsilon^{\mathrm{rel}}\|Y^l\|_F,
$$
where $R^l= X^l-Z^l$, $S^l = \delta(Z^l-Z^{l-1})$ are primal and dual residuals respectively at the $l$-th iteration. $\epsilon^{\mathrm{abs}}>0$ is an absolute tolerance and $\epsilon^{\mathrm{rel}}>0$ is a relative tolerance, and they are both algorithm parameters.
$\del$ is typically fixed, but one can also adaptively update it during iterations following the rule in \cite{ADMM}; for instance,
\begin{equation*}
\del^{l+1} =
\begin{cases}
2\del^l & \mbox{if} \quad \|R^l\|_F>10\|S^l\|_F,\\
\del^l/2 & \mbox{if} \quad 10\|R^l\|_F<\|S^l\|_F,\\
\del^l & \mbox{otherwise.}
\end{cases}
\end{equation*}

\subsection{Real-valued, nonnegative signals.}
If the signal is known to be real or nonnegative, we should add one more constraint to the complex PhaseLiftOff (\ref{unl1l2}):
\begin{equation}\label{unl1l2R}
\min_{X\in\C^{n\times n}} \; \varphi(X) \quad \mbox{s.t.} \quad X\succeq 0, \; X\in\Omega.
\end{equation}
Here $\Omega$ is $\R^{n\times n}$ (or resp., $\R^{n\times n}_{+}$), which means each entry of $X$ is real (or resp., nonnegative). Thus we need to modify the DCA (\ref{iter}) accordingly:
\begin{equation}\label{iterR}
X^{k+1} =
\begin{cases}
\arg\min_{X\in\C^{n\times n}} \; \frac{1}{2}\|\A(X) - b\|_2^2 + \lam\Tr(X) \quad \mbox{s.t.} \quad X\succeq0, \; X\in\Omega & \mbox{if}\quad X^k = 0,\\
\arg\min_{X\in\C^{n\times n}} \; \frac{1}{2}\|\A(X) - b\|_2^2 + \lam\langle X, I_n - \frac{X^k}{\|X^k\|_F}\rangle \quad \mbox{s.t.} \quad X\succeq0, \; X\in\Omega & \mbox{otherwise}.\\
\end{cases}
\end{equation}
The above subproblem at each DCA iteration can also be solved by ADMM. Specifically, we want to solve the optimization problem of the following form:
\begin{equation}\label{sub}
\min_{X\in\C^{n\times n}} \; \hf\|\A(X) - b\|_2^2 + \la X,W\ra \quad \mbox{s.t.} \quad X \succeq 0, \; X\in\Omega.
\end{equation}
In ADMM form, (\ref{sub}) is reformulated as
$$
\min_{X\in\C^{n\times n}} \; \hf\|\A(X) - b\|_2^2 + \la X,W\ra  + g_{\Omega}(X) + g_{\succeq}(Z) \quad \mbox{s.t.} \quad X - Z = 0,
$$
where $g_{\succeq}(Z)$ is the same as in (\ref{lag}), and
\begin{equation*}
g_{\Omega}(X) =
\begin{cases}
0 & \mbox{if} \quad X\in\Omega,\\
\infty & \mbox{otherwise}.
\end{cases}
\end{equation*}
Having defined the augmented Lagrangian
$$
\mathcal{L}_{\del}(X,Y,Z) = \hf\|\A(X) - b\|_2^2 + \la X,W\ra + \la Y,X-Z \ra + \frac{\del}{2}\|X-Z\|_F^2 + g_{\Omega}(X) + g_{\succeq}(Z),
$$
we arrive at Algorithm \ref{ADMM1} by alternately minimizing $\mathcal{L}_{\del}$ with respect to $X$, minimizing with respect to $Z$, and updating the dual variable $Y$.
\begin{algorithm}
\caption{ADMM for solving (\ref{sub})}
\begin{algorithmic}\label{ADMM1}
\WHILE {not converged}
\STATE $X^{l+1} = \P_{\Omega}((\A^*\A + \del \I_n)^{-1}(\A^*(b)- W + \del Z^l -Y^l))$
\STATE $Z^{l+1} = \P_{\succeq}(X^{l+1} + Y^l/\del)$
\STATE $Y^{l+1} = Y^l + \del(X^{l+1} - Z^{l+1})$
\ENDWHILE
\end{algorithmic}
\end{algorithm}
The operator $\P_{\Omega}:\H^{n\times n} \rightarrow \Omega$ in Algorithm \ref{ADMM1} represents the projection onto the set $\Omega$. In particular, $\P_{\Omega}(X) = \mathrm{Re}(X)$ is the real part of $X$ for $\Omega = \R^{n\times n}$, whereas $\P_{\Omega}(X) = \max\{\mathrm{Re}(X),0\}$ for $\Omega = \R_{+}^{n\times n}$.
Algorithm \ref{ADMM1} is almost identical to Algorithm \ref{ADMM} except that an extra projection $\P_{\Omega}$ is performed in the $X$-update step.

\section{Numerical Experiments.}\label{numerical}
In this section, we report numerical results. Besides the proposed (\ref{unl1l2}) and the regularized PhaseLift (\ref{unl1}), we also discuss the following reweighting scheme from \cite{CESV}, which is an extension of reweighted $\ell_1$ algorithm in the regime of compressed sensing introduced in \cite{rwl1}:
\be\label{rwsub}
X^{k+1} = \arg\min_{X\in\C^{n\times n}} \; \hf\|\A(X) - b\|_2^2 + \lam\la W^k, X\ra \quad \mbox{s.t.} \quad X \succeq 0,
\ee
where $W^0 = I_n$ and $W^k = (X^k+\varepsilon I_n)^{-1}$ for $k\geq1$ and for some $\varepsilon>0$. The aim of this scheme is to provide more accurate solutions with lower rank than that of PhaseLift.
Note that $W^k$ is exactly the gradient of $\log(\det(X+\varepsilon I_n))$ at $X^k$, the reweighting scheme is in essence an implementation of the DCA attempting to solve the nonconvex problem
\be\label{reweighted}
\min_{X\in\C^{n\times n}} \; \hf\|\A(X) - b\|_2^2 + \lam\log(\det(X+\varepsilon I_n)) \quad \mbox{s.t.} \quad X \succeq 0.
\ee
Here the DC components are $\hf\|\A(X) - b\|_2^2$ and $-\lam\log(\det(X+\varepsilon I_n))$. We hereby remark that with positive semidefinite constraint, the DCA is basically equivalent to the reweighting scheme. In \cite{CESV}, (\ref{unl1}) and the subproblem (\ref{rwsub}) of (\ref{reweighted}) are solved by FISTA.
Here we solve them using ADMM (Algorithm \ref{ADMM}) instead as we find it more efficient.

\subsection{Exact recovery from noise-free measurements.}
We set up a phase retrieval problem by 1) generating a random complex-valued
signal $\hat{x}$ of length $n=32$ whose real and imaginary parts are Gaussian, 2) sampling a Gaussian
matrix $A\in\C^{n\times m}$ with $m = 60, 62, \dots, 150$, and 3)
computing the measurements $b = \A(\hat{x}\hat{x}^*)$.  We then solve (\ref{unl1l2}), (\ref{unl1}) and (\ref{reweighted}) to get approximations to $\hat{x}\hat{x}^*$. The ultimate goal of phase retrieval is to reconstruct the signal $\hat{x}$ rather than the rank-1 matrix $\hat{x}\hat{x}^*$. So given a solution $\tX$, we need to compute the relative mean squared error (rel. MSE) between $\tx =\sqrt{\sigma_1(\tX)}u_1$ and $\hat{x}$ modulo a global phase term to measure the recovery quality, where $\sigma_1(\tX)$ is the largest singular value (or eigenvalue) of $\tX$ and $u_1$ the corresponding unit-normed eigenvector. More precisely, the rel. MSE is given by
$$
\min_{c\in\C:|c| = 1}\frac{\|c\tx - \hat{x}\|_2^2}{\|\hat{x}\|_2^2}.
$$
It is easy to show that its minimum occurs at
$$\tilde{c} = \frac{\la \tx,\hat{x} \ra}{|\la \tx,\hat{x} \ra|}.$$
A recovery is considered as a success if the rel. MSE is less than $10^{-6}$ (or equivalently, relative error $< 10^{-3}$). For each $m = 60, 63, \dots, 150$, we repeat the above procedures 100 times and record the success rate for each model.

For (\ref{unl1}), we set $\lam = 10^{-4}$, $\epsilon^{\mathrm{rel}}=10^{-5}$ and $\epsilon^{\mathrm{abs}}= 10^{-7}$ in its ADMM algorithm; for (\ref{unl1l2}), $\lam = 10^{-4}$, $\epsilon^{\mathrm{rel}}=10^{-5}$, $\epsilon^{\mathrm{abs}}= 10^{-7}$ and $\verb|tol| = 10^{-2}$; parameters for (\ref{reweighted}) are the same as those for (\ref{unl1l2}) except that there is an additional parameter $\varepsilon=2$.
In addition, the maximum iteration set for all ADMM algorithms is 5000 and that for the DCA and the reweighting algorithm are both 10. All three methods start with the same initial point $X^0 = 0$.

The success rate v.s. number of measurements plot is shown in Figure \ref{fig:success}.
The result validates that nonconvex proxy for the rank functional gives significantly
better recovery quality than the convex trace norm. A similar finding has been reported
in the regime of compressed sensing \cite{YLHX}.
We also observe that PhaseLiftOff outperforms $\log$-$\det$ regularization.
This is not surprising as the former always captures rank-1 solutions.
In Figure \ref{fig:success}, one can see that when the number of measurements
is $m\approx3n=96$, solving our model by the DCA guarantees exact recovery with high probability.
Recall that in theory \cite{BCE} at least $3n-2$ measurements are needed to recover the signal exactly.
This is an indication that the proposed method is likely to provide the optimal
practical results one can hope for.

\begin{figure}\centering
\includegraphics[width=0.80\textwidth,height = 0.40\textwidth]{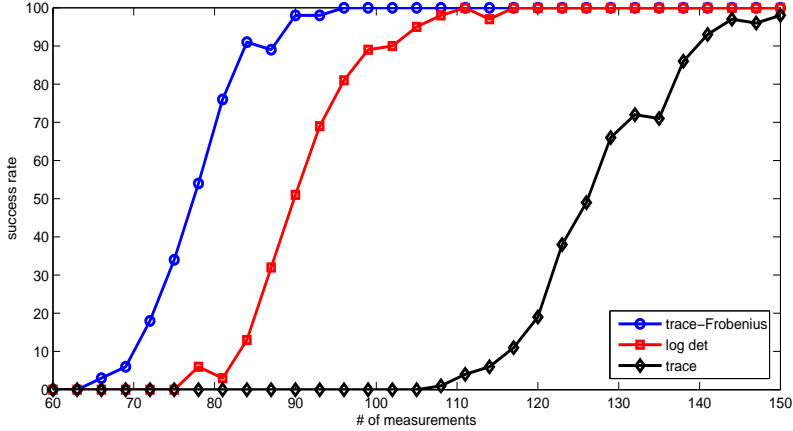}
\caption{success rate v.s. number of measurements with parameters $n = 32$, $m = 60, 63, \dots, 150$, and
100 runs at each $m$.}\label{fig:success}
\end{figure}

\subsection{Robust recovery from noisy measurements.}\label{noisycase}
We investigate how the proposed method performs in the presence of
noise. The test signal $\hat{x}$ is a Gaussian complex-valued signal of length $n = 32$.
We sample $m = 4n$ Gaussian measurement vectors in $\C^n$ and compute the
measurements $b\in\R^m$, followed by adding additive white Gaussian noise
by means of the MATLAB function $\verb|awgn(b,snr)|$. There are 6 noise levels
varying from 5dB to 55dB. We then apply the DCA to achieve a reconstruction $\tX$ and
compute the signal-to-noise ratio (SNR) of reconstruction in dB defined as $-10\log_{10}(\mbox{rel. MSE}).$
The SNR of reconstruction for each noise level is finally averaged over 10 independent runs.

A crucial point to address here is how we set the value of $\lam$. Theorem \ref{equivalence} predicts that
provided the noise amount $\|e\|_2$ is known,
when $\lam > \frac{\|\A\|\|e\|_2}{\sqrt{2}-1}\approx 2.414\|\A\|\|e\|_2$,
the PhaseLiftOff (\ref{unl1l2}) is equivalent to the phase retrieval problem (\ref{energymin}),
and its solution is no longer related to $\lam$. From computational perspective, however, $\lam$ cannot be
too large as the algorithm may often get stuck at a local solution.
An extreme example is that if $\lam$ is exceedingly large,
the DCA will be trapped at the initial guess $X^0 = 0$.
On the other hand, if $\lam$ is too small, the reconstruction will be of course far
from the ground truth as $\A(X)=b$ tends to be enforced.
But can we choose $\lam$ that is less than $2.414\|\A\|\|e\|_2$? The answer is yes, since this bound only provides a sufficient condition for equivalence.

Suppose the noise amount $\|e\|_2$ (or its estimate) is known, defining
$$
\mu := \|\A\|\|e\|_2 = \sqrt{\|A^*A\circ\overline{A^*A}\|_2}\|e\|_2,
$$
we try 4 different values of  $\lam$ in each single run. They are multiples of $\mu$, namely $0.01\mu$, $0.2\mu$, $2.5\mu$ and $50\mu$. The maximum outer and inner iterations are 10 and 5000 respectively. The other parameters are $\epsilon^{\mathrm{rel}}=10^{-5}$, $\epsilon^{\mathrm{abs}}= 10^{-7}$, $\verb|tol| = 10^{-2}$.
The reconstruction results are depicted in Figure \ref{fig:SNR}.
The two curves for $0.2\mu$ and $2.5\mu$ nearly coincide, and they are almost linear, which
strongly suggest stable recoveries. In contrast, the algorithm with $0.01\mu$ and $50\mu$ performed poorly.
Although $\lam = 50\mu$ yields comparable reconstruction when there is little noise,
the DCA clearly encounters local minima in the low SNR regime. On the other hand, $0.01\mu$ is too small.
Summarizing these observations, we conclude that for the DCA method, a reasonable value for $\lam$ lies
in the interval (but not limited to) [$0.2\mu$,$2.5\mu$].

\begin{figure}\centering
\includegraphics[width=0.80\textwidth,height = 0.40\textwidth]{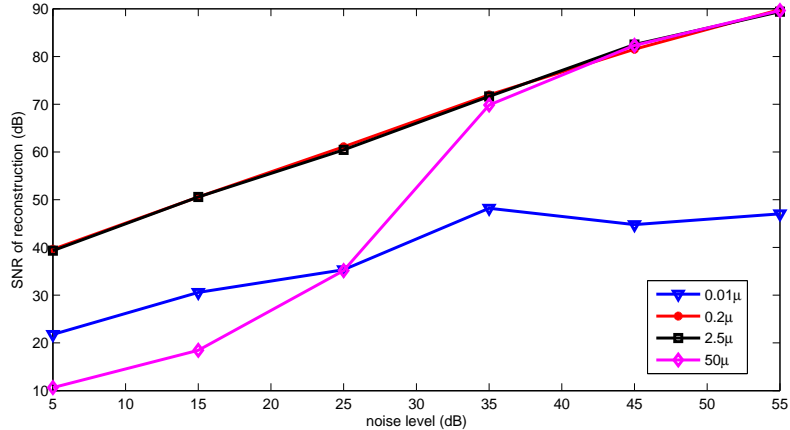}
\caption{SNR of signal recovery v.s. noise level in measurement (in SNR dB).
Parameters are: $n = 32$, noise level = 5dB, 15dB, $\dots$, 55dB;
$\lam = 0.01\mu, 0.2\mu, 2.5\mu, 50\mu$; with 10 runs at each noise level.}\label{fig:SNR}
\end{figure}

\section{Conclusions.}
We introduced and analyzed a novel penalty (trace minus Frobenius norm) for phase retrieval
in the PhaseLiftOff least squares regularization problem.
We proved its equivalence with rank-1 least squares
and stable recovery for noisy measurement at high probability.
The DC algorithm for energy minimization is proved to converge to
a stationary point satisfying KKT conditions without imposing strict convexity of convex components of the energy
(a step beyond the standard DCA theory \cite{TA_97}). Numerical experiments showed that the PhaseLiftOff method outperforms
PhaseLift and its nonconvex variant ($\log$-$\det$ regularization). The minimal number of measurements for exact
recovery by PhaseLiftOff approaches the theoretical limit. In future work, we shall further explore the potential
of PhaseLiftOff in phase retrieval applications and rank-1 optimization problems. We are also interested in developing faster optimization algorithms that are more robust to the value of $\lam$.

\medskip

{\bf Acknowledgements.} The work was partially supported by NSF grant DMS-1222507.
We thank the March 2014 NSF Algorithm Workshop in Boulder, CO, and Dr. E. Esser for communicating
recent developments in phase retrieval research.

\end{document}